\documentclass[11pt]{amsart}
\usepackage{amssymb, amsmath, amsthm, amsfonts, amscd}
\usepackage{xcolor}
\usepackage{datetime}
\usepackage{enumerate}
\newtheorem{theorem}{Theorem}[section]
\newtheorem{lemma}[theorem]{Lemma}
\newtheorem{proposition}[theorem]{Proposition}
\newtheorem{corollary}[theorem]{Corollary}
\theoremstyle{definition}
\newtheorem{definition}[theorem]{Definition}
\newtheorem{example}[theorem]{Example}

\newtheorem{note}[theorem]{Note}

\theoremstyle{remark}
\newtheorem{remark}[theorem]{Remark}
\numberwithin{equation}{section}

\title{Weyl's theorem for  paranormal closed operators}

\author{Neeru Bala}
\address{Department of Mathematics, Indian Institute of Technology - Hyderabad, Kandi, Sangareddy, Telangana, India 502 285.}
\email{ma16resch11001@iith.ac.in}

\author{G. Ramesh}
\address{Department of Mathematics, Indian Institute of Technology - Hyderabad, Kandi, Sangareddy, Telangana, India 502 285.}
\email{rameshg@iith.ac.in}

	\subjclass[2010]{47A10, 47A53; 47B20}
	\keywords{Closed operator, Fredholm operator, minimum modulus, paranormal operator, Riesz projection, Weyl's theorem}
\date{\currenttime ;  \today}
\begin{document}

	\begin{abstract}
In this article we discuss  a few spectral properties of a paranormal  closed operator (not necessarily bounded) defined in a Hilbert space. This class contains closed symmetric operators. First we show that  the spectrum of such an operator is non empty. Next, we  give a characterization of closed range operators in terms of the spectrum.  Using these results we prove the Weyl's theorem: if $T$ is a densely defined closed, paranormal operator, then $\sigma(T)\setminus\omega(T)=\pi_{00}(T)$, where $\sigma(T), \omega(T)$ and $\pi_{00}(T)$ denote the spectrum, Weyl spectrum and the set of all isolated eigenvalues with finite multiplicities, respectively. Finally, we prove that  the Riesz projection $E_\lambda$ with respect to any isolated spectral value $\lambda$ of $T$ is self-adjoint and satisfies $R(E_\lambda)=N(T-\lambda I)=N(T-\lambda I)^*$.
	\end{abstract}

	\maketitle

\section{Introduction}
	One of the most important and well studied class in operator theory is the class of normal operators. The spectral theorem for normal operators assures the existance of non-trivial invariant subspace and also reveals the complete structure of the operator. Thus normal operators led to several generalizations, one among such generalizations is the class of paranormal operators.
	
The class of bounded paranormal operators was first studied by Istr$\mathrm{\breve{a}}$tescu \cite{IST1}, who named it as the \textit{class N}. Further, Furuta \cite{FUR} introduced the term paranormal operator.

	Bounded paranormal operators are studied by many authours, for eg \cite{ANDO,FUR,IST1,IST,UCH}. In particular, Ando \cite{ANDO} gave a characterization of bounded paranormal operators. Istr$\mathrm{\breve{a}}$tescu \cite{IST1} proved that the class of \textit{normaloid operators} is a generalization of paranormal operators. We have the following inclusion relation between some subclasses and generalized class of bounded paranormal operators.
	$$Normal\subseteq Hyponormal\subseteq Paranormal \subseteq Normaloid.$$
	 The above inclusion relations are proper. For more delatils, we refer to \cite{FUR,HAL}. The definition of bounded paranormal operators is extended to unbounded operators by Daniluk \cite{DEN}, where he has discussed about closability of unbounded paranormal operators.
	
	 In this article we are going to deal with densely defined closed paranormal operators in a Hilbert space $H$ and prove the following results.
	
	 Let $T$ be a densely defined closed paranormal operator in $H$. Then
	 \begin{enumerate}
	 	\item\label{3} Spectrum of $T$ is non-empty.
	 	\item\label{1} Every isolated spectral value of $T$ is an eigenvalue.
	 	\item\label{2} In addition, if $N(T)=N(T^*)$, then
	 	\begin{enumerate}
	 		\item Range of $T$ is closed if and only if $0$ is an isolated spectral value of $T$.
	 		\item The minimum modulus, $m(T)$ is equal to the distance of $0$ from spectrum of $T$.
	 	\end{enumerate}
 \item\label{4} $T$ satisfies the Weyl's Theorem \textit{i.e.} $\sigma(T)\setminus\omega(T)=\pi_{00}(T)$. Here $\omega(T)$ is the Weyl's spectrum and $\pi_{00}(T)$ consists of all isolated eigenvalues with finite multiplicity.
	\item\label{5} If $\lambda$ is an isolated spectral value of $T$, then the Riesz projection $E_\lambda$ with respect to $\lambda$ is self-adjoint and satisfies $R(E_\lambda)=N(T-\lambda I)=N(T-\lambda I)^*$. 	
	 \end{enumerate}
	
	  Results (\ref{1}) and (\ref{2}) are well known in the literature for self-adjoint  operators. For unbounded self-adjoint operators, a simple proof of these results is given by Kulkarni \textit{et. al.}\cite{KUL1}, without using the spectral theorem. For the bounded case, three elementary proofs are given in \cite{KUL}. Also (\ref{3}) is well known for self-adjoint and normal operators, refer \cite[Lemma 8.6, Page 102]{HEL} for more details.
	
	  Weyl's theorem and self-adjointness of Riesz projection with respect to isolated spectral value of an operator is studied for many different class of operators. For some non-normal operators (hyponormal and Toeplitz operators), this was established by Coburn \cite{COB}. Further Uchiyama \cite{UCH}, extended it to bounded paranormal operators using Ando's characterization \cite{ANDO} for paranormal operators. But, since Ando's characterization is not available for unbounded paranormal operators, the techniques of bounded operators does not work in our case. Hence  we try to prove (\ref{4}) and (\ref{5}), using a different approach.
	
 This article is divided into four sections. In second section we set up some notations and known results which we will be using throughout the article. In the third section we discuss some spectral properties of densely defined closed paranormal operators. In the last section we prove Weys'l theorem for paranormal operators.
	\section{Notations and preliminaries}
	In this article we consider complex Hilbert spaces, which will be denoted by $H, H_1, H_2$ etc. The inner product and the induced norm are denote by $\langle,\rangle$ and $\|.\|$, respectively. Given any two Hilbert spaces $H_1$ and $H_2$, their Cartesian product is defined as
	$$H_1\times H_2:=\{(h_1,h_2):h_1\in H_1,h_2\in H_2\}.$$
	It is a Hilbert space with the following inner product
	$$\langle(h_1,h_2),(k_1,k_2)\rangle:=\langle h_1,k_1\rangle_{H_1}+\langle h_2,k_2\rangle_{H_2}$$ for all $h_1,k_1\in H_1$ and $h_2,k_2\in H_2$.
	
	The space of all linear operators from $H_1$ to $H_2$ is denoted by $\mathcal{L}(H_1,H_2)$. We write $\mathcal{L}(H,H)=\mathcal{L}(H)$. For $T\in\mathcal{L}(H_1,H_2)$, the domain, null space and range space of $T$ are denoted by $D(T)$, $N(T)$ and $R(T)$, respectively. If $\overline{D(T)}=H_1$, then $T$ is called a \textit{densely defined operator}. The space $C(T):=D(T)\cap N(T)^{\perp}$ is called the carrier of $T$.
	
	An operator $T\in\mathcal{L}(H_1,H_2)$ is said to be \textit{bounded}, if there exist a positive real number $M$ such that $\|Tx\|\leq M\|x\|$ for all $x\in D(T)$. The space of all bounded linear operators from $H_1$ to $H_2$ is denoted by $\mathcal{B}(H_1,H_2)$. In particular, $\mathcal{B}(H):=\mathcal{B}(H,H)$.
	
	If $M$ is a  subspace of $H_1$, then $T|_{M}$ denotes the restriction of $T$ to $M$ and $S_{M}:=\{x\in M:\|x\|=1\}$ will denote the unit sphere in $M$. We denote the identity operator on $M$ by $I_M$. If $T\in\mathcal{L}(H)$ and $M$ is a closed subspace of $H$, then $M$ is said to be invariant under $T$, if for every $x\in D(T)\cap M$, $Tx$ is in $M$.
	
	For $T\in\mathcal{L}(H_1,H_2)$, \textit{Graph} of $T$ is defined by
	$$\mathcal{G}(T):=\{(x,Tx):x\in D(T)\}\subseteq H_1\times H_2.$$
	\begin{definition}
		A linear operator $T$ from $H_1$ to $H_2$ is called \textit{closed} if its graph $\mathcal{G}(T)$ is a closed subspace of the Hilbert space $H_1\times H_2$.
		
		Equivalently, $T$ is said to be closed, if for any sequence $(x_n)\subseteq D(T)$ with $x_n\rightarrow x$ and $Tx_n\rightarrow y$ implies $x\in D(T)$ and $Tx=y$.
	\end{definition}
    By the closed graph Theorem, a closed linear operator defined on the whole space is bounded. It follows that domain of any unbounded closed operator is a proper subspace of a Hilbert space. The class of all closed linear operators from $H_1$ to $H_2$ is denoted by $\mathcal{C}(H_1,H_2)$. We write $\mathcal{C}(H,H)=\mathcal{C}(H)$.

    For every densely defined operator $T\in\mathcal{L}(H_1,H_2)$, there exists a unique operator $T^*\in\mathcal{L}(H_2,H_1)$, which satisfies

    $$\langle Tx,y\rangle=\langle x,T^*y\rangle,\,\forall x\in D(T),\,y\in D(T^*),$$
    where
    $D(T^*)=\{y\in H_2:x\rightarrow\langle Tx,y\rangle\text{ is continuous on }D(T)\}.$ This operator $T^*$ is called the \textit{adjoint} of $T$. The denseness of domain is necessary and sufficient for the existence of the adjoint.

    If $S$ and $T$ are two closed operators, then $S$ is called an \textit{extension} of $T$ (or $T$ is a \textit{restriction} of $S$), if $D(T)\subseteq D(S)$ and $Sx=Tx$ for all $x\in D(T)$. This is often denoted as $T\subset S$. Consequently $S=T$ if and only if $D(S)=D(T)$ and $Sx=Tx$ for all $x\in D(S)=D(T)$.

     A densely defined operator $T\in\mathcal{C}(H)$ is said to be \textit{self-adjoint} if $T^*=T$ and \textit{normal} if $TT^*=T^*T$. If $T\in\mathcal{B}(H)$, then $T$ is \textit{hyponormal} if $TT^*\leq T^*T$. Equivalently, $T$ is hyponormal if $\|T^*x\|\leq\|Tx\|$ for all $x\in H$.

	Analogous to the bounded operator we can define the minimum modulus of closed operator.
	\begin{definition}\cite{GOL,SHKGR}
		Let $T\in\mathcal{C}(H_1,H_2)$. Then
		\begin{enumerate}
			\item the \textit{minimum modulus} of $T$ is defined by $m(T):=\inf\{\|Tx\|:x\in S_{D(T)}\}$.\\
			\item The \textit{reduced minimum modulus} of $T$ is defined by $\gamma(T):=\inf\{\|Tx\|:x\in S_{C(T)}\}$.
		\end{enumerate}
	\end{definition}
	By the definition, it is clear that $m(T)\leq\gamma(T)$.
	
	If $T\in\mathcal{C}(H_1,H_2)$ is a densely defined operator and $N(T)=\{0\}$, then the inverse operator, $T^{-1}$ is the linear operator from $H_2$ to $H_1$, with $D(T^{-1})=R(T)$ and $T^{-1}(Tx)=x$ for all $x\in D(T)$. In particular if $T\in \mathcal C(H)$ is densely defined and bijective, then by the closed graph theorem it follows that $T^{-1}\in \mathcal B(H)$. In addition if  $T\in\mathcal{C}(H)$ is normal then $T$ has a bounded inverse if and only if $m(T)>0.$
	
   Here we provide some results related to closed range operators that we will need later.
	\begin{theorem}\cite{BEN}\label{thm4}
		For a densely defined operator $T\in\mathcal{C}(H_1,H_2)$, the following are equivalent.
		\begin{enumerate}
			\item $R(T)$ is closed .
			\item $R(T^*)$ is closed.
			\item $\gamma(T)>0$.
			\item $T_0=T|_{C(T)}$ has a bounded inverse.
		\end{enumerate}
	\end{theorem}
\begin{lemma}\cite[Lemma 3.3]{KUL2}\label{lemcarrier}
	Let $T\in\mathcal{C}(H_1,H_2)$ be densely defined. Then $\overline{C(T)}=N(T)^{\perp}.$
\end{lemma}
\begin{definition}
	Let $T\in\mathcal{C}(H)$. Then the \textit{resolvent set} of $T$ is defined by
	$$\rho(T)=\{\lambda\in\mathbb{C}:T-\lambda I\text{ is invertible and }(T-\lambda I)^{-1}\in\mathcal{B}(H)\}$$
	and $\sigma(T)=\mathbb{C}\setminus\rho(T)$ is called the \textit{spectrum} of $T$.
\end{definition}
If $T\in\mathcal{C}(H)$, then $\sigma(T)$ is a closed subset of $\mathbb{C}$. Moreover $\sigma(T)$ can be empty set or whole $\mathbb{C}$. Refer to \cite{REED} for more details. The spectrum of $T$ decomposes as the disjoint union of the \textit{point spectrum} $\sigma_p(T)$, \textit{continuous spectrum} $\sigma_c(T)$ and \textit{residual spectrum} $\sigma_r(T)$, where
\begin{align*}
\sigma_p(T)=&\{\lambda\in\mathbb{C}:T-\lambda I\text{ is not injective}\}\\
\sigma_r(T)=&\{\lambda\in\mathbb{C}:T-\lambda I\text{ is injective but }R(T-\lambda I)\text{ is not dense in }H\}\\
\sigma_c(T)=&\sigma(T)\setminus(\sigma_p(T)\cup\sigma_r(T)).
\end{align*}
The \textit{spectral radius} of $T\in\mathcal{B}(H)$ is defined by
$$r(T):=\sup\{|\lambda|:\lambda\in\sigma(T)\}.$$
An operator $T\in\mathcal{B}(H)$ is said to be \textit{normaloid}, if $r(T)=\|T\|.$
\begin{definition}\cite{SCE}
	A densely defined operator $T\in\mathcal{C}(H)$ is called \textit{Fredholm operator} if $R(T)$ is closed, $\dim(N(T))$ and $\dim(R(T)^{\perp})$ are finite.\\
	In this case $ind(T)=\dim(N(T))-\dim(R(T)^{\perp})$ is called the \textit{index} of $T$.
\end{definition}
\begin{remark}\label{rem3}
	If $T,K\in\mathcal{C}(H)$ are Fredholm and compact operator, respectively then $T+K$ is also Fredholm and $ind(T+K)=ind(T).$
\end{remark}
 For more details about Fredholm operators, refer \cite{SCE}.
\begin{definition}\cite{BAX}
	 Let $T\in\mathcal{C}(H)$. Then the \textit{Weyl's spectrum} of $T$ is defined by
	 $$\omega(T)=\{\lambda\in\mathbb{C}:T-\lambda I \text{ is not Fredholm operator of index 0}\}$$
	 	and $\pi_{00}(T)=\{\lambda\in\sigma_p(T): \lambda\text{ is isolated with }\dim(N(T-\lambda I))<\infty\}.$
	
\end{definition}

	Suppose $T\in\mathcal{C}(H)$ with $\sigma(T)=\sigma\cup\tau$, where $\sigma$ is contained in some bounded domain $\Delta$ such that $\bar{\Delta}\cap\tau=\emptyset$. Let $\Gamma$ be the boundary of $\Delta$, then
\begin{equation}\label{eqn3}
E_{\sigma}=\frac{1}{2\pi i}\int_{\Gamma}(zI-T)^{-1}dz,
\end{equation}
is called the \textit{Riesz projection} with respect to $\sigma$.
\begin{theorem}\cite[Theorem 2.1, Page 326]{GOH}\label{thmgoh}
	Suppose $T\in\mathcal{C}(H)$ with $\sigma(T)=\sigma\cup\tau$, where $\sigma$ is contained in some bounded domain $\Delta$ and $E_{\sigma}$ is the operator defined in Equation \ref{eqn3}. Then
	\begin{enumerate}
		\item $E_{\sigma}$ is a projection.
		\item The subspace $R(E_{\sigma})$ and $N(E_{\sigma})$ are invariant under $T$.
		\item The subspace $R(E_{\sigma})$ is contained in $D(T)$ and $T|_{R(E_{\sigma})}$ is bounded.
		\item $\sigma(T|_{R(E_{\sigma})})=\sigma\text{ and }\sigma(T|_{N(E_{\sigma})})=\tau$.
	\end{enumerate}
\end{theorem}

In particular, if $\lambda$ is an isolated point of $\sigma(T)$, then there exist a positive real number $r$ such that $\{z\in\mathbb{C}:|z-\lambda|\leq r\}\cap\sigma(T)=\{\lambda\}$. If we take $\Gamma$ to be the boundary of $\{z\in\mathbb{C}:|z-\lambda|\leq r\}$, we define the Riesz projection with respect to $\lambda$ as
\begin{equation}\label{eqn4}
E_{\lambda}=\frac{1}{2\pi i}\int_{\Gamma}(zI-T)^{-1}dz.
\end{equation}
For more details about Riesz projection, see \cite{GOH,LOR}.
\begin{definition}\cite[Definition 1.1]{DEN}
	An operator $T\in\mathcal{L}(H)$ is called paranormal operator if
	\begin{equation}\label{eqn2}
	\|Tx\|^2\leq\|T^2x\|\|x\|,\,\forall\,x\in D(T^2).
	\end{equation}
	Equivalently, $T$ is paranormal, if $\|Tx\|^2\leq\|T^2x\|,\,\forall\,x\in S_{D(T^2)}$.
\end{definition}
If $T\in\mathcal{B}(H)$, then Equation \ref{eqn2} holds for every $x\in H$. More details about bounded paranormal operators can be found in \cite{ANDO,FUR,IST,KUB,UCH}. Here we will summarize some well known results for paranormal operators.
\begin{lemma}\cite{KUB}\label{lem3}
	Let $T\in\mathcal{B}(H)$ be paranormal operator and $M$ be a closed subspace of $H$, which is invariant under $T$. Then $T|_{M}$ is also paranormal.
\end{lemma}
\begin{theorem}\cite[Theorem 1]{IST}\label{thm6}
	If $T\in\mathcal{B}(H)$ is paranormal, then
	\begin{enumerate}
		\item $T$ is normaloid.
		\item $T^{-1}$ is paranormal, if $T$ is invertible.
		\item If $\sigma(T)$ lies on the unit circle, then $T$ is unitary operator.
	\end{enumerate}
\end{theorem}

	\section{Spectral properties}
	In this section, we will study some spectral properties of densely defined closed paranormal operators.
	
	A densely defined linear operator $T$ is called \textit{symmetric}, if $T\subseteq T^*$, that is $D(T)\subseteq D(T^*)$ and $Tx=T^*x$, for all $x\in D(T)$. 
	
	Next we show that every symmetric operator is paranormal.
	\begin{proposition}\label{symm}
	Let $T\in\mathcal{C}(H)$ be a symmetric operator. Then $T$ is paranormal.
	\end{proposition}
	\begin{proof}
   For every $x\in D(T^2)$, we have
	\begin{align*}
	\|Tx\|^2=&\langle T^*Tx,x\rangle\\
	=&\langle T^2x,x\rangle\leq\|T^2x\|\|x\|.
	\end{align*}	
	This proves the result.
	\end{proof}
	 Next we generalize result (2) of Theorem \ref{thm6} to paranormal closed operators.
	\begin{proposition}\label{propinverse}
		Let $T\in\mathcal{C}(H)$ be a densely defined paranormal operator. If $0\notin\sigma(T)$ then $T^{-1}$ is paranormal.
	\end{proposition}
	\begin{proof}
		Since $T^{-1}$ exists, we get $R(T^2)=H$. As $T$ is paranormal, it is easy to observe that $N(T^2)=N(T)$, so $T^2$ is bijective and $(T^2)^{-1}$ exists. Also
		\begin{align*}
		D((T^{-1})^2)=&\{x\in D(T^{-1}):T^{-1}x\in D(T^{-1})\}\\
		=&\{x\in H:T^{-1}x\in H\}\\
		=&H=R(T^2).
		\end{align*}
		If $y\in H$, then there exist some $x\in D(T^2)$, such that $y=T^2x.$
		Now
		\begin{align*}
		\|T^{-1}y\|^2=\|Tx\|^2\leq&\|T^2x\|\|x\|\\
		=&\|y\|\|T^{-2}y\|.
		\end{align*}
		Hence $T^{-1}$ is paranormal.
	\end{proof}
It is well known that spectrum of a densely defined closed normal operator is non-empty. Here we will prove this result for  the  class of paranormal operators.
	\begin{theorem}\label{thm2}
		If $T\in\mathcal{C}(H)$ be a densely paranormal operator, then $\sigma(T)$ is non-empty.
	\end{theorem}
	\begin{proof}
		On the contrary, assume that $\sigma(T)=\emptyset$. Thus $T$ is invertible and $T^{-1}\in\mathcal{B}(H)$.
		
		First we will show that $\sigma(T^{-1})=\{0\}$. For any complex number $\lambda\neq 0$, consider the operator $S=\lambda^{-1}T(T-\lambda^{-1}I)^{-1}$. Since $S$ can also be written as the sum of two bounded operators, $S=\lambda^{-1}(I+\lambda^{-1}(T-\lambda^{-1}I)^{-1})$, so $S$ is bounded. Clearly $S$ is the bounded inverse of $T^{-1}-\lambda I$. Thus $\sigma(T^{-1})\subseteq\{0\}$. As $T^{-1}\in\mathcal{B}(H)$, $\sigma(T^{-1})$ is non-empty, we conclude that $\sigma(T^{-1})=\{0\}$.
		
		By Proposition \ref{propinverse}, $T^{-1}$ is bounded paranormal operator, thus normaloid by Theorem \ref{thm6}. Hence $\|T^{-1}\|=0$, which implies $T^{-1}=0$, a contradiction. Hence $\sigma(T)$ is non-empty.
	\end{proof}
	Note that in Theorem \ref{thm2}, we only used the fact that $T^{-1}$ is normaloid. Thus we can make the following statement.
	\begin{proposition}
		If $T\in\mathcal{C}(H)$ be densely defined closed operator such that $T^{-1}$ is normaloid, then $\sigma(T)\ne\emptyset$.
	\end{proposition}

	As we know from Lemma \ref{lem3} that restriction of a bounded paranormal operator to an invariant subspace is paranormal. On the similar lines we can prove the following Lemma.
	\begin{lemma}\label{invalem}
		Let $T\in\mathcal{C}(H)$ be a paranormal operator. Suppose $M$ is a closed subspace of $H$ which is invariant under $T$, then $T|_{M}$ is paranormal.
	\end{lemma}
\begin{proof}
	As $M$ is invariant under $T$, we have
	\begin{align*}
	D(T^2|_M)=&D(T^2)\cap M\\
	=&\{x\in D(T):Tx\in D(T)\}\cap M\\
	=&\{x\in D(T)\cap M:Tx\in D(T)\cap M\},\; \text{since}\; T(D(T)\cap M)\subseteq M\\
	=&\{x\in D(T|_M):Tx\in D(T|_M)\}\\
	=&D((T|_M)^2).
	\end{align*}
	Thus we get $T^2|_M=(T|_M)^2$. Now the result follows from the following inequality;
	$$\|T|_Mx\|^2=\|Tx\|^2\leq\|T^2x\|=\|T^2|_Mx\|=\|(T|_M)^2x\|,\forall x\in S_{D((T|_M)^2)}.$$
\end{proof}
Next we will show that every isolated spectral value of paranormal operator is eigenvalue.
	\begin{lemma}\label{nullrange}
		Let $T\in\mathcal{C}(H)$ be densely defined and $\lambda$ be an isolated point of $\sigma(T)$. Then $N(T-\lambda I)\subseteq R(E_{\lambda})$, where $E_{\lambda}$ is the Riesz projection with respect to $\lambda$ defined in Equation \ref{eqn4}.
	\end{lemma}
	\begin{proof}
		Let us consider
		\begin{equation*}
		S:=\frac{1}{2\pi i}\int_{\Gamma}(z-\lambda)^{-1}(zI-T)^{-1}dz,
		\end{equation*}
		where $\Gamma$ is the boundary of the disc $D=\{z\in\mathbb{C}:|z-\lambda|\leq r\}$ such that $D\cap\sigma(T)=\{\lambda\}.$
		For any $z\in\rho(T)$,
		\begin{equation}\label{eqn5}
		\begin{split}
		(z-\lambda)^{-1}(zI-T)^{-1}(T-\lambda I)=&(z-\lambda)^{-1}(zI-T)^{-1}[T-zI+zI-\lambda I]\\
		=&-(z-\lambda)^{-1}I_{D(T)}+(zI-T)^{-1}I_{D(T)}.
		\end{split}
		\end{equation}
		As $$-\int_{\Gamma}(z-\lambda )^{-1}I_{D(T)}dz+\int_{\Gamma}(z-T)^{-1}I_{D(T)}dz$$ is well defined, so is $S(T-\lambda I)$. Integrating Equation \ref{eqn5} on $\Gamma$, we get $S(T-\lambda I)=-I_{D(T)}+E_{\lambda}|_{D(T)}$.
		If we take any $x\in N(T-\lambda I)$, then $(-I_{D(T)}+E_{\lambda}|_{D(T)})x=0$. Consequently $x=E_{\lambda}x\in R(E_{\lambda})$. Hence $N(T-\lambda I)\subseteq R(E_{\lambda})$.
	\end{proof}
	\begin{proposition}\label{nulrangeprop}
		Let $T\in\mathcal{C}(H)$ be a densely defined paranormal operator. If $\lambda$ is an isolated point of $\sigma(T)$, then $N(T-\lambda I)=R(E_{\lambda})$.
	\end{proposition}
	\begin{proof}
		By Lemma \ref{nullrange}, $N(T-\lambda I)\subseteq R(E_{\lambda})$. To complete the proof we have to show that $R(E_{\lambda})\subseteq N(T-\lambda I)$.
		
		By Theorem \ref{thmgoh} and Lemma \ref{invalem}, $T|_{R(E_{\lambda})}$ is bounded and paranormal. Then by Theorem \ref{thm6}, $T|_{R(E_{\lambda})}$ is normaloid.
		
		If $\lambda=0$, then $\sigma(T|_{R(E_{0})})=\{0\}$. This implies $\|T|_{R(E_0)}\|=0$, so we get $T|_{R(E_0)}=0$. Hence $R(E_0)\subseteq N(T).$
		
		Next if $\lambda\neq 0$, then $\sigma({\lambda}^{-1}T|_{R(E_{\lambda})})=\{1\}$. By Theorem \ref{thm6}, it follows that ${\lambda}^{-1}T|_{R(E_{\lambda})}$ is unitary. Thus $T|_{R(E_{\lambda})}-\lambda I_{R(E_{\lambda})}$ is normal and $\sigma(T|_{R(E_{\lambda})}-\lambda I_{R(E_{\lambda})})=\{0\}$. Since every normal operator is normaloid, we conclude that $T|_{R(E_{\lambda})}-\lambda I_{R(E_{\lambda})}=0$. Hence $R(E_{\lambda})\subseteq N(T-\lambda I)$.
	\end{proof}
	\begin{note}
		Proposition \ref{nulrangeprop} is proved for bounded paranormal operators by Uchiyama \cite{UCH}.
	\end{note}
\begin{corollary}\label{coroRE}
	Let $T$ be as defined in Proposition \ref{nulrangeprop} and $\lambda$ be an isolated point of $\sigma(T)$. Then $N(E_\lambda)=R(T-\lambda I).$
\end{corollary}
\begin{proof}
	By Theorem \ref{thmgoh}, $\lambda\notin\sigma(T|_{N(E_\lambda)})$, thus $R((T-\lambda I)|_{N(E_\lambda)})=N(E_\lambda)$ and $R((T-\lambda I)|_{N(E_\lambda)})\subseteq R(T-\lambda I)$, consequently $N(E_\lambda)\subseteq R(T-\lambda I)$.
	
	If $y\in R(T-\lambda I)$, then there exist $x\in D(T)$ such that $y=(T-\lambda I)x$. Since $H=R(E_\lambda)\oplus N(E_\lambda)$, so $x$ can be written as
	$$x=u+v,\text{ where }u\in R(E_\lambda),\,v\in N(E_\lambda).$$
	By Proposition \ref{nulrangeprop}, $u\in N(T-\lambda I)\subseteq D(T)$, then $v=x-u\in D(T)$ and by Theorem \ref{thmgoh} $N(E_\lambda)$ is invariant under $T$, we obtain
	$$y=(T-\lambda I)x=(T-\lambda I)v\in N(E_\lambda).$$
	Hence $R(T-\lambda I)\subseteq N(E_\lambda)$. This proves the result.
\end{proof}
	Next we give a characterization of closed range paranormal operators.
	\begin{lemma}\label{lemisopt}
		Suppose $T\in\mathcal{C}(H)$ is a densely defined paranormal operator. If $0$ is an isolated point of $\sigma(T)$, then $R(T)$ is closed.
	\end{lemma}
	\begin{proof}
		Since $0$ is an isolated point of $\sigma(T)$, we can consider the Riesz projection $E_0$ with respect to $0$. By Theorem \ref{thmgoh} we get $0\notin\sigma(T|_{N(E_0)})$ and Corollary \ref{coroRE} says $R(T)=R(T|_{N(E_0)})$, which is closed. This proves the result.
	\end{proof}
	In general the converse of Lemma \ref{lemisopt} is not true. We have the following example to illustrate this.
	\begin{example}\label{eg1}
		Let $T:\ell^2(\mathbb N)\rightarrow \ell^2(\mathbb N)$ be defined by
		$$T(x_1,x_2,\ldots)=(0,x_1,x_2,\ldots),\text{ for all } (x_n)\in \ell^2(\mathbb{N}).$$
		Then $\sigma(T)=\{z\in\mathbb{C}:|z|\leq 1\}$, $R(T)=\ell^2(\mathbb N)\setminus\overline{span}\{e_1\}.$ Hence $R(T)$ is closed but $0$ is not an isolated point of $\sigma(T).$
	\end{example}
	 Next we will give a sufficient condition under which the converse of Lemma \ref{lemisopt} is also true.
	\begin{theorem}\label{thmisopt}
		Let $T\in\mathcal{C}(H)$ be a densely defined paranormal operator with $N(T)=N(T^*)$ and $0\in\sigma(T)$. Then $0$ is an isolated point of $\sigma(T)$ if and only if $R(T)$ is closed.
	\end{theorem}
	\begin{proof}
		Forward implication followed by Lemma \ref{lemisopt}.
		
		For the reverse implication, assume that $R(T)$ is closed. Consider $T_0=T|_{N(T)^{\perp}}:N(T)^{\perp}\cap D(T)\rightarrow N(T)^{\perp}.$ Clearly $T_0$ is injective and $R(T_0)=R(T)$ is closed. Also $R(T_0)=N((T^*)^{\perp})=N(T)^{\perp}$, consequently $T_0$ is bijective and $T_0^{-1}\in\mathcal{B}(N(T)^{\perp})$. Thus $0\notin\sigma(T_0)$. Applying \cite[Theorem 5.4, Page 289]{TAL}, $\sigma(T)\subseteq\{0\}\cup\sigma(T_0).$ Since $0\in\sigma(T)$, $\sigma(T)=\{0\}\cup\sigma(T_0)$, thus $0$ is an isolated point of $\sigma(T)$.
	\end{proof}
	Note that Theorem \ref{thmisopt} does not hold if we drop the condition $N(T)=N(T^*)$. Consider the operator $T$ defined in Example \ref{eg1}. Clearly $N(T)=\{0\}\ne\overline{span}\{e_1\}=N(T^*)$, also Theorem \ref{thmisopt} does not hold for $T$.

	\begin{theorem}\label{thmminmod}
		Let $T\in\mathcal{C}(H)$ be a densely defined paranormal operator. If $N(T)=N(T^*)$, then $m(T)=d(0,\sigma(T))$, the distance between $0$ and $\sigma(T).$
	\end{theorem}
	\begin{proof}
		We will prove this result by considering the following two cases, which exhaust all the possiblities.\\
		Case (1): $T$ is not injective. Clearly $m(T)=0$ and $0\in\sigma_{p}(T)$. Hence $m(T)=0=d(0,\sigma(T)$.\\
		Case (2): $T$ is injective. It suffices to show that $\gamma(T)=d(0,\sigma(T))$ because $m(T)=\gamma(T)$.
		
		First assume that $\gamma(T)=0$. Then by Theorem \ref{thm4}, $R(T)$ is not closed, consequently $0\in\sigma_c(T)$. Thus $d(0,\sigma(T))=0=\gamma(T).$
		
		Next assume that $\gamma(T)>0$. By Theorem \ref{thm4}, we get $R(T)$ is closed. Note that $0\notin\sigma(T)$, because if $0\in\sigma(T)$, then by Theorem \ref{thmisopt} and Proposition \ref{nulrangeprop}, $0\in\sigma_p(T)$. But this is not true, as $T$ is injective. Thus $0\notin\sigma(T)$ and $T^{-1}$ is bounded paranormal operator, by Proposition \ref{propinverse}. Consequently $T^{-1}$ is normaloid, by Theorem \ref{thm6}. Hence by \cite[Proposition 2.12]{KUL1},
		\begin{align*}
		\gamma(T)=&\frac{1}{\|T^{-1}\|}\\
		=&\frac{1}{r(T^{-1})}\\
		=&\frac{1}{\sup\{|\lambda|:\lambda\in\sigma(T^{-1})\}}\\
		=&\inf\{|\delta|:\delta\in\sigma(T)\}\\
		=&d(0,\sigma(T)).
		\end{align*}
		This completes the proof.
	\end{proof}
As a consequence of Theorem \ref{thmminmod} we have the following result.
\begin{corollary}\label{cordist}
	Let $T\in\mathcal{C}(H)$ be a densely defined paranormal operator. If $N(T)=N(T^*)$, then $\gamma(T)=d(T):=\inf\{|\lambda|:\lambda\in\sigma(T)\setminus\{0\}\}.$
\end{corollary}
\begin{proof}
	Consider the operator $T_0=T|_{N(T)^{\perp}}:C(T)\rightarrow N(T)^{\perp}$. By Lemma \ref{invalem} and Theorem \ref{thmminmod}, $T_0$ is paranormal and
	$$\gamma(T)=m(T_0)=d(0,\sigma(T_0))=d(T).$$
	This proves the result.
\end{proof}
	\begin{remark}
		Theorem \ref{thmminmod} does not hold if $N(T)\ne N(T^*)$. The following example illustrates this fact.
	\end{remark}
\begin{example}\label{eg3}
	Let $T:\ell^2(\mathbb N)\rightarrow \ell^2(\mathbb N)$ be defined by
	$$T(x_1,x_2,x_3,\ldots)=(0,x_1,2x_2,3x_3,\ldots)$$
	where $D(T)=\{(x_1,x_2,x_3,\ldots)\in \ell^2(\mathbb N):\sum_{i=1}^{\infty}\|ix_i\|^2<\infty\}$.
	
	As $C_{00}$, the space of all complex sequences consisting of atmost finitely many non zero terms is a subset of $D(T)$ and $C_{00}$ is dense in $\ell^2(\mathbb N)$,  $T$ is densely defined. It is easy to see that $T$ is a closed operator. Thus $T^*$ is well defined and
	\begin{equation*}
	T^*(x_1,x_2,x_3,\ldots)=(x_2,2x_3,3x_4,\ldots)
	\end{equation*}
with $D(T^*)=\{x\in \ell^2(\mathbb N):\sum_{i=2}^{\infty}\|(i-1)x_i\|^2<\infty\}.$
	
	For any $x=(x_n)\in D(T^2)$ we have,
	\begin{align*}
	\|T(x)\|^2=&\sum_{i=1}^{\infty}\|ix_i\|^2\\
	\leq&\sum_{i=1}^{\infty}(i+1)i\|x_i\|^2\\
	\leq&\sqrt{\sum_{i=1}^{\infty}\left((i+1)i\|x_i\|\right)^2}\sqrt{\sum_{i=1}^{\infty}\|x_i\|^2}\\
	=&\|T^2x\|\|x\|.
	\end{align*}
	Hence $T$ is paranormal.
	
	Since $\|Tx\|\geq\|x\|$ for all $x\in D(T)$ and $\|Te_1\|=\|e_1\|$, we get $m(T)=1$.
    Also it can be easily verified that $T$ is injective, $R(T)=\ell^2(\mathbb N)\setminus span\{e_1\}$ is closed but $R(T)\ne H$, so $0\in\sigma(T)$. Hence $d(0,\sigma(T))=0\ne 1=m(T).$

    Now we will show that $\sigma(T)=\mathbb{C}$. To prove this first we will show that $T-\lambda I$ is injective and $N(T-\lambda I)^*\neq\{0\},$ for all $\lambda\in\mathbb{C}$.

    Let $\lambda \in \mathbb C\setminus {\{0}\}$ and $(T-\lambda I)x=0$ for some $x=(x_n)\in D(T)$. Then
    $$(-\lambda x_1,x_1-\lambda x_2,2x_2-\lambda x_3,\ldots)=0.$$
    Equating component-wise we get $x=0$, thus $T-\lambda I$ is injective.

    Let $y=(y_n)\in D(T^*)$ be such that $(T-\lambda I)^*y=0$. This implies
    $$(y_2-\bar{\lambda}y_1,2y_3-\bar{\lambda}y_2,3y_4-\bar{\lambda}y_3,\ldots)=0.$$
    From this we get
    \begin{equation}\label{eqn7}
    y=\left(1,\bar{\lambda},\frac{\bar{\lambda}^2}{2!},\frac{\bar{\lambda}^3}{3!},\ldots\right)y_1.
    \end{equation}
    If $\lambda=0$, then $N(T^*)=\overline{span}\{e_1\}.$ If $\lambda\neq 0$, then we will show that $y$ obtained in Equation \ref{eqn7} belongs to $N(T-\lambda I)^*$. Consider
 $ z_{n}=\frac{\overline{\lambda}^{2n}}{(n!)^{2}}$, then
 \begin{equation*}
 \left|\frac{z_{n+1}}{z_n}\right|=\frac{|\lambda|^2}{(n+1)^2}\rightarrow 0 \text{ as }n\rightarrow\infty.
 \end{equation*}
 By the ratio test we conclude that $\sum_{n=1}^{\infty}z_n$ is absolutely convergent. That is $\sum_{i=0}^{\infty}\left(\frac{|\lambda|^n}{n!}\right)^2<\infty.$ Thus $y\in \ell^2(\mathbb N)$. On the similar lines we can show that $\sum_{i=1}^{\infty}\left(\frac{|\lambda|^n}{(n-1)!}\right)^2<\infty.$ Hence $N(T-\lambda I)^*\neq \{0\}$.

 For every $\lambda\in\mathbb{C}$, $N(T-\lambda I)=\{0\}$ and $\overline{R(T-\lambda I)}=(N(T-\lambda I)^*)^{\perp}\ne l_2(\mathbb{N})$. Thus we conclude that $\lambda\in\sigma_r(T)$, hence $\sigma(T)=\mathbb{C}$.

 We also have $\gamma(T)=1\ne 0=d(T)$. From this we can conclude that Corollary \ref{cordist} is also not true if the condition, $N(T)=N(T^*)$ is dropped.
\end{example}
\begin{remark}
It is well known that the residual spectrum of any closed densely defined normal operator is empty. But this is not true in the case of paranormal operators, as in Example \ref{eg3} the residual spectrum of $T$ is whole $\mathbb{C}$.
\end{remark}
\section{Weyl's theorem for paranormal operators}
In this section we show that a densely defined closed paranormal operator $T$ satisfy Weyl's theorem. We also prove that the Riesz projection  $E_\lambda$ with respect to any isolated spectral value $\lambda$ of $T$ is self-adjoint.

If $H=H_1\oplus H_2$ is a Hilbert space and $T\in\mathcal{C}(H)$, then $T$ has the block matrix representation
\begin{equation}\label{eqn8}
T=
\begin{bmatrix}
T_{11}&T_{12}\\
T_{21}&T_{22}
\end{bmatrix},
\end{equation}
where $T_{ij}:D(T)\cap H_j\rightarrow H_i$ is defined by $T_{ij}=P_{H_i}TP_{H_j}|_{D(T)\cap H_j}$ for $i,j=1,2$. Here $P_{H_i}$ is an orthogonal projection onto $H_i$.

 For $(x_1,x_2)\in (H_1\cap D(T))\oplus (H_2\cap D(T))$,  $$T(x_1,x_2)=(T_{11}x_1+T_{12}x_2,T_{21}x_1+T_{22}x_2).$$
 Note that if $T$ is densely defined then $T_{ij}$ is densely defined for $i,j=1,2$, that is $\overline{D(T_{ij}\cap H_j)}=H_j$ for all $i,j=1,2$.

\begin{remark}\label{rem2}
	Let $T$ be as defined in Equation \ref{eqn8}. If $H_1=N(T)\ne\{0\}$ and $H_2=N(T)^{\perp}$, then
	\begin{equation}\label{eqn9}
	T=
	\begin{bmatrix}
	0&T_{12}\\
	0&T_{22}
	\end{bmatrix}.
	\end{equation}
	\begin{enumerate}
		\item If $T$ is densely defined closed operator then $T_{22}$ is also densely defined closed operator.
		\item\label{2rem2} It can be easily checked that $R(T_{22})=R(T)\cap N(T)^{\perp}.$ If $R(T)$ is closed, then $R(T_{22})$ is closed.
		
	\end{enumerate}
\end{remark}

Any $T\in\mathcal{C}(H)$ is said to satisfy the Weyl's theorem if the Weyl's spectrum, $\omega(T)$ consists of all spectral values of $T$ except the isolated eigenvalues of finite multiplicity.  That is $\sigma(T)\setminus\omega(T)=\pi_{00}(T)$.

 In \cite{COB}, Coburn proved that any bounded hyponormal and Toeplitz operator satisfies the Weyl's theorem. This was extended by Uchiyama  \cite{UCH}  to bounded paranormal operators. Here we are going to prove this for unbounded paranormal operators.
\begin{theorem}\label{weylthm}
	Let $T\in\mathcal{C}(H)$ be a densely defined paranormal operator. Then $\sigma(T)\setminus\omega(T)=\pi_{00}(T)$.
\end{theorem}
\begin{proof}
	Let $\lambda\in\sigma(T)\setminus\omega(T)$. So, we have $\dim(N(T-\lambda I))=\dim(N(T-\lambda I)^*)<\infty$ and $R(T-\lambda I)$ is closed.
	
	Then $T-\lambda I$ can be decomposed on $H=N(T-\lambda I)\oplus N(T-\lambda I)^{\perp}$ as
	\[
	T-\lambda I=
	\begin{bmatrix}
		0 & T_{12}\\
		0 & T_{22}-\lambda I_{N(T-\lambda I)^{\perp}}
	\end{bmatrix},
	\]
	where $T_{22}=P_{N(T-\lambda I)^{\perp}}T|_{N(T-\lambda I)^{\perp}}$. By Remark \ref{rem2}, $T_{22}-\lambda I_{N(T-\lambda I)^{\perp}}$ is a densely defined closed operator with domain $C(T-\lambda I)$ and $R(T_{22}-\lambda I_{N(T-\lambda I)^{\perp}})$ is closed.
	
	As $N(T-\lambda I)$ is finite dimensional, $T_{12}$ is finite rank operator and $ind(T_{12})=0$, by Remark \ref{rem3}. Thus $ind(T-\lambda I)=ind\left(N(T_{22}-\lambda I_{N(T-\lambda I)^{\perp}})\right)=0$.
	
	Since $N(T_{22}-\lambda I_{N(T-\lambda I)^{\perp}})=\{0\}$, we get $N(T_{22}-\lambda I_{N(T-\lambda I)^{\perp}})^*=\{0\}$ and consequently $\overline{R(T_{22}-\lambda I_{N(T-\lambda I)^{\perp}})}=N(T-\lambda I)^{\perp}$. Thus $T_{22}-\lambda I_{N(T-\lambda I)^{\perp}}$ has bounded inverse and $\lambda\notin\sigma(T_{22})$. As $\sigma(T)\subseteq\{\lambda\}\cup\sigma(T_{22})$, $\lambda$ is an isolated point of $\sigma(T)$. Hence $\lambda\in \pi_{00}(T)$.
	
	Conversely, let $\lambda\in\pi_{00}(T)$. Now consider the Riesz projection $E_\lambda$ with respect to $\lambda$, as defined in Equation \ref{eqn4}. By Theorem \ref{thmgoh} and Corollary \ref{coroRE}, $\lambda\notin\sigma(T|_{N(E_\lambda)})$ and
	\begin{align*}
	R(T-\lambda I)=&R\left((T-\lambda I)|_{N(E_\lambda)}\right)\\
	=&N(E_\lambda).
	\end{align*}
	 As $\lambda\notin\sigma(T|_{N(E_\lambda)})$, this implies $R((T-\lambda I)|_{N(E_{\lambda}})$ is closed and so is $R(T-\lambda I)$. Also $((T-\lambda I)|_{N(E_\lambda)})^{-1}\in\mathcal{B}(N(E_\lambda))$, thus we get
	 \begin{align*}
	 \dim N(T-\lambda I)^*=&\dim (R(T-\lambda I)^{\perp})\\
	 =&\dim (N(E_\lambda)^{\perp})\\
	 =&\dim (R(E_\lambda))\\
	 =&\dim (N(T-\lambda I)).
	 \end{align*}
	 Hence $T-\lambda I$ is Fredholm operator of index zero. This proves our result.
\end{proof}
As a consequence of Theorem \ref{weylthm} and Proposition \ref{symm}, we have the following result.
\begin{corollary}
	Let $T\in\mathcal{C}(H)$ be a symmetric operator. Then $T$ satisfies the Weyl's theorem.
\end{corollary}
\begin{theorem}\label{thmself}
	Let $T\in\mathcal{C}(H)$ be a densely defined paranormal operator and $\lambda$ be an isolated point of $\sigma(T)$. Then the Riesz projection $E_{\lambda}$ with respect to $\lambda$ satisfies
	\begin{equation*}
	R(E_{\lambda})=N(T-\lambda I)=N(T-\lambda I)^*.
	\end{equation*}
	Moreover $E_\lambda$ is self-adjoint.
\end{theorem}
\begin{proof}
	By Theorem \ref{thmgoh} and Corollary \ref{coroRE}, $\lambda\notin N(E_\lambda)$ and $R(T-\lambda I)=N(E_\lambda)$. As $T|_{N(T-\lambda I)^{\perp}}$ is the bijection from $N(T-\lambda I)^{\perp}\cap D(T)$ to $R(T-\lambda I)$, thus we get $N(E_\lambda)\cap D(T)\subseteq N(T-\lambda I)^{\perp}\cap D(T)$.
	
	Now we claim that $N(E_\lambda)\cap D(T)=N(T-\lambda I)^{\perp}\cap D(T)$. Let $x\in N(T-\lambda I)^{\perp}\cap D(T)$ and
	$$E_\lambda x=u+v,\text{ where }u\in N(T-\lambda I),\,v\in N(T-\lambda I)^{\perp}.$$
	Operating $E_\lambda$ on both sides, we get
	$$u+v=E_\lambda x=u+E_\lambda v.$$
	This implies $E_\lambda v=v\in R(E_\lambda )\cap N(T-\lambda I)^{\perp}=\{0\}$, by Proposition \ref{nulrangeprop}. From this we conclude that $E_\lambda x=u=E_\lambda u$, that is $\,x-u\in N(E_\lambda)\cap D(T)\subseteq N(T-\lambda I)^{\perp}\cap D(T).$ As $x\in N(T-\lambda I)^{\perp}$, we get $u\in N(T-\lambda I)\cap N(T-\lambda I)^{\perp}=\{0\}$. Consequently $E_\lambda x=0$, thus $ N(T-\lambda I)^{\perp}\cap D(T)\subseteq N(E_\lambda)\cap D(T) $. Hence $N(T-\lambda I)^{\perp}\cap D(T)=N(E_\lambda )\cap D(T)$.
	
	By Lemma \ref{lemcarrier} and Corollary \ref{coroRE}, we get
	\begin{align*}
	N(T-\lambda I)^{\perp}=&\overline{N(T-\lambda I)^{\perp}\cap D(T)}\\
	=&\overline{R(T-\lambda I)\cap D(T)}\\
	=&\overline{(N(T-\lambda I)^*)^{\perp}\cap D(T)}\\
	\subseteq&(N(T-\lambda I)^*)^{\perp}.
	\end{align*}
	Hence $N(T-\lambda I)^*\subseteq N(T-\lambda I)$. By Corollary \ref{coroRE}, we have $N(E_\lambda)^{\perp}\subseteq R(E_\lambda)$.
	
	Let $x\in R(E_\lambda)$, then $x=a+b$ where $a\in N(E_\lambda)$ and $b\in N(E_\lambda )^{\perp}$. As $N(E_\lambda)^{\perp}\subseteq R(E_\lambda)$, we get $a=x-b\in N(E_\lambda)\cap R(E_\lambda)=\{0\}$. Thus we get $N(E_\lambda)^{\perp}=R(E_\lambda)$, which is equivalent to say that $N(T-\lambda I)=N(T-\lambda I)^*$.
	
	As $N(E_\lambda)^{\perp}=R(E_\lambda)$, $E_\lambda$ is an orthogonal projection. Hence $E_\lambda$ is self-adjoint.
\end{proof}
\begin{corollary}
	Let $T\in\mathcal{C}(H)$ be a densely defined paranormal operator. If $\lambda_1$ and $\lambda_2$ are two distinct isolated points of $\sigma(T)$, then $N(T-\lambda_1I)$ is orthogonal to $N(T-\lambda_2I)$.
\end{corollary}
\begin{proof}
	Without loss of generality, assume that $\lambda_1\ne 0$. For any $x\in N(T-\lambda_1I)$ and $y\in N(T-\lambda_2I)$, we have
	\begin{align*}
	\langle x,y\rangle=&{\lambda_1}^{-1}\langle\lambda_1x,y\rangle\\
	=&\lambda_1^{-1}\langle Tx,y\rangle.
	\end{align*}
	By Theorem \ref{thmself}, $N(T-\lambda_2I)=N(T-\lambda_2I)^*$. Thus we get
	\begin{align*}
	\langle x,y\rangle=&\lambda_1^{-1}\langle x, T^*y\rangle\\
	=&\lambda_1^{-1}\langle x,\overline{\lambda_2}y\rangle\\
	=&\lambda_1^{-1}\lambda_2\langle x,y\rangle.
	\end{align*}
	This implies either $\langle x,y\rangle=0$ or $\lambda_1^{-1}\lambda_2=1$.
	As $\lambda_1\ne\lambda_2$, we get $\langle x,y\rangle=0$. This proves the result.
\end{proof}

\end{document}